\documentclass[reqno]{amsart}

\usepackage{amsfonts, upgreek}
\usepackage{amssymb, amsopn, latexsym,amsmath,graphics,verbatim,amsthm,enumerate,amscd,tabularx}
\usepackage[all,cmtip]{xy}
\usepackage{graphicx}
\usepackage{fancyhdr}
\usepackage[active]{srcltx}
\usepackage[british]{babel}
\usepackage{amsmath,amssymb,amsthm,mathrsfs, url}

\theoremstyle{plain}
\newtheorem*{theorem*}{Theorem}
\newtheorem*{conjecture*}{Conjecture}

\newtheorem{thm}{Theorem}[section]
\newtheorem{theorem}[thm]{Theorem}

\newtheorem{corollary}[thm]{Corollary}

\newtheorem*{lemma*}{Lemma}

\newtheorem{proposition}[thm]{Proposition}
\theoremstyle{definition}
\newtheorem{definition}[thm]{Definition}
\theoremstyle{remark}
\newtheorem*{remark}{Remark}

\input cyracc.def \font\tencyr=wncyr10 \def\russe{\tencyr\cyracc}
\def\Sha{\text{\russe{Sh}}}

\newdir{ (}{{}*!/-5pt/@^{(}}

\SelectTips{eu}{10}
\UseTips
\entrymodifiers={+!!<0pt,\fontdimen22\textfont2>}

\DeclareMathOperator{\Gal}{Gal}
\DeclareMathOperator{\rank}{rank}
\DeclareMathOperator{\corank}{corank}

\DeclareMathOperator{\Selm}{Sel}

\DeclareMathOperator{\img}{img}

\DeclareMathOperator{\Hom}{Hom}

\DeclareMathOperator{\res}{res}

\DeclareMathOperator{\cor}{cor}
\DeclareMathOperator{\infl}{inf}
\DeclareMathOperator{\dual}{dual}
\DeclareMathOperator{\Frob}{Frob}

\newcommand{\Q}{{\mathbb{Q}}}
\newcommand{\Z}{{\mathbb{Z}}}

\newcommand{\C}{{\mathbb{C}}}

\newcommand{\Fp}{{\mathbb{F}_p}}
\newcommand{\Zp}{{\mathbb{Z}_p}}

\newcommand{\ilim}{\mathop{\varprojlim}\limits}
\newcommand{\dlim}{\mathop{\varinjlim}\limits}
\newcommand{\Sel}{{\Selm_p}}
\newcommand{\Selinf}{{\Selm_{p^{\infty}}}}

\newcommand{\cN}{\mathcal{N}}
\newcommand{\cG}{\mathcal{G}}
\newcommand{\cO}{\mathcal{O}}

\newcommand{\cy}[1]{\mathbb{Z}/#1\mathbb{Z}}

\newcommand{\overbar}[1]{\mkern 1.5mu\overline{\mkern-1.5mu#1\mkern-1.5mu}\mkern 1.5mu}

\newcommand{\joinrelshort}{\mathrel{\mkern-9mu}}
\newcommand{\shortlongrightarrow}{\relbar\joinrelshort\rightarrow}
\newcommand{\isomarrow}{\mathrel{\mathop{\setbox0\hbox{$\mathsurround0pt
        \shortlongrightarrow$}\ht0=0.7\ht0\box0}\limits
        ^{\sim\mkern2mu}}}

\begin{document}

\title{Selmer Groups and Anticyclotomic $\Zp$-extensions II}

\author{Ahmed Matar}

\address{Department of Mathematics\\
         University of Bahrain\\
         P.O. Box 32038\\
         Sukhair, Bahrain}
\email{amatar@uob.edu.bh}

\begin{abstract}
Let $E/\Q$ be an elliptic curve, $p$ a prime where $E$ has ordinary reduction and $K_{\infty}/K$ the anticyclotomic $\Zp$-extension of a quadratic imaginary field $K$ satisfying the Heegner hypothesis. We give sufficient conditions on $E$ and $p$ in order to ensure that $\Selinf(E/K_{\infty})$ is a cofree $\Lambda$-module of rank one. We also show that these conditions imply that $\rank(E(K_n))=p^n$ and that $\Sha(E/K_n)[p^{\infty}]=\{0\}$ for all $n \geq 0$.
\end{abstract}

\maketitle

\section{Introduction}
Let $K$ be an imaginary quadratic field with discriminant $d_K \neq -3,-4$ whose class number we will denote by $h_K$.

Let $E$ an elliptic curve of conductor $N$ defined over $\Q$ with a modular parametrization $\pi: X_0(N) \to E$ which maps the cusp $\infty$ of $X_0(N)$ to the origin of $E$ (see \cite{Wiles} and \cite{BDCT}).

Assume that every prime dividing $N$ splits in $K/\Q$. It follows that we can choose an ideal $\cN$ such that $\cO_K/\cN \cong \Z/N\Z$. Therefore the natural projection of complex tori:
$$\C/\cO_K \to \C/\cN^{-1}$$\\
is a cyclic $N$-isogeny, which corresponds to a point of $x_1 \in X_0(N)$. The theory of complex multiplication shows that $x_1$ is rational over $K_1$, the Hilbert class field of $K$. Let $y_1=\pi(x_1) \in E(K_1)$ and define the point $y_K=\text{Tr}_{K_1/K}(y_1) \in E(K)$. Kolyvagin's celebrated paper \cite{Kolyvagin} proves that when $y_K$ has infinite order, then $E(K)$ has rank 1 and the Tate-Shafarevich group $\Sha(E/K)$ is finite.

In this paper, we work with a particular prime $p$ and therefore are interested in the following weaker result of Kolyvagin (see \cite{Gross})

\begin{theorem*}[Kolyvagin]
Let $p$ be an odd prime such that $\Gal(\Q(E[p])/\Q)=GL_2(\Fp)$ and such that $p$ does not divide $y_K$ in $E(K)$, then $E(K)$ has rank 1 and $\Sha(E/K)[p^{\infty}]=\{0\}$.
\end{theorem*}

The main theorem of this article can be thought of as an extension of the above result of Kolyvagin to the tower fields of the anticylotomic $\Zp$-extension of $K$. Before stating the result let us list the hypotheses we will work under.

Let $p \geq 5$ be a prime. We shall say that $(E,p)$ satisfies $(\star)$ if the following are met:
\begin{enumerate}[(i)]
\item All the primes dividing $N$ split in $K/\Q$
\item $p$ does not divide $Nd_Kh_K \cdot \#E(\Fp) \cdot \prod_{v | N} c_v$
\item $\Gal(\Q(E[p])/\Q)=GL_2(\Fp)$
\item $E$ has ordinary reduction at $p$
\item $a_p \not \equiv -1 \; (\text{mod } p)$ if $p$ is inert in $K/\Q$
\item $a_p \not \equiv 2 \; (\text{mod } p)$ if $p$ splits in $K/\Q$
\end{enumerate}

In the above, $c_v$ is the Tamagawa number of $E$ at the prime $v$ of $K$ and the product $\prod_{v | N} c_v$ runs over all primes of $K$ dividing $N$. $a_p$ is the integer $1+p-\#E(\Fp)$.

Note that the above conditions are not identical to the conditions in \cite{Matar1} and \cite{Matar2}.

Let $K_{\infty}/K$ be the anticyclotomic $\Zp$-extension of $K$, $\Gamma=\Gal(K_{\infty}/K)$ and $K_n$ the unique subfield of $K_{\infty}$ containing $K$ such that $\Gal(K_n/K) \cong \cy{p^n}$. Denote $\Gamma_n=\Gamma^{p^n}$, $G_n=\Gamma/\Gamma_n$ and $R_n=\Fp[G_n]$.

For any $n$ and $m$ we let $\Selm_{p^m}(E/K_n)$ denote the $p^m$-Selmer group of $E$ over $K_n$ defined by
$$\displaystyle 0 \longrightarrow \Selm_{p^m}(E/K_n) \longrightarrow H^1(K_n, E[p^m])\longrightarrow \prod_v H^1(K_{n,v}, E)[p^m].$$
We also define the $p^{\infty}$-Selmer group of $E$ over $K_n$ as $\Selinf(E/K_n)=\dlim\Selm_{p^m}(E/K_n).$

Finally we define the $p^m$-Selmer group and the $p^{\infty}$-Selmer group of $E$ over $K_{\infty}$ as $\Selm_{p^m}(E/K_{\infty})=\dlim\Selm_{p^m}(E/K_n)$ and $\Selinf(E/K_{\infty})=\dlim\Selinf(E/K_n)$.

Let $\Lambda=\Zp[[\Gamma]]$ be the Iwasawa algebra attached to $K_{\infty}/K$. Fixing a topological generator $\gamma \in \Gamma$ allows us to identify $\Lambda$ with the power series ring $\Zp[[T]]$. Throughout most of the paper we work ``mod $p$'' and so we will also consider the ``mod $p$'' Iwasawa algebra $\overbar{\Lambda}=\Lambda/p\Lambda=\Fp[[T]]$.

For any discrete torsion abelian group $A$ we let $A^{\dual}=\Hom(A, \Q/\Z)$ denote in Pontryagin dual. The main result of this article is the following theorem

\begin{theorem*}
Assume that $(E,p)$ satisfies $(\star)$ and that $p$ does not divide $y_K$ in $E(K)$. Then we have
\begin{enumerate}[(i)]
\item $\Selinf(E/K_{\infty})^{\dual}$ is a free $\Lambda$-module of rank 1
\item $\rank(E(K_n))=p^n$ for all $n \geq 0$
\item $\Sha(E/K_n)[p^{\infty}]=\{0\}$ for all $n \geq 0$.
\end{enumerate}
\end{theorem*}

The method of proof of this theorem is a modification of the method used in \cite{Matar1}. To explain this, let $\ell$ be a rational prime that is relatively prime to $pNd_K$ and such that $\Frob_{\ell}(K(E[p])/\Q)=[\tau]$ where $\tau$ is a complex conjugation on the algebraic closure of $\Q$. Note that $\ell$ is inert in $K$ and splits completely in $K_{\infty}/K$.

For any $n$, we define $E(K_{n,\ell})/p:=\oplus_{\lambda|\ell}E(K_{\lambda,n})/p$ and
$H^1(K_{n,\ell}, E)[p]:=\oplus_{\lambda|\ell}H^1(K_{n,\lambda}, E)[p]$.

Tate local duality gives a nondegenerate pairing

$$\langle \; , \; \rangle_{\ell}: E(K_{n,\ell})/p \times H^1(K_{n,\ell}, E)[p] \to \Fp.$$\\
This identifies $H^1(K_{n, \ell}, E)[p]$ with $(E(K_{n,\ell})/p)^{\dual}$.

Taking the direct limit of the groups $H^1(K_{n, \ell}, E)[p]$ with respect to restriction and the inverse limit of the groups $E(K_{n, \ell})/p$ with respect to corestriction (norm) we get an isomorphism

$$\dlim H^1(K_{n,\ell}, E)[p] \cong (\ilim E(K_{n,\ell})/p)^{\dual}.$$\\
Now consider the $p$-Selmer group $\Sel(E/K_n)$. If $s \in \Sel(E/K_n)$, then for any prime $v$ of $K_n$ the image of $s$ in $H^1(K_{n,v}, E[p])$ (under restriction) belongs to $E(K_{n,v})/p$. Therefore, the restriction map induces a map

$$\res_{\ell}: \Sel(E/K_n) \to E(K_{n,\ell})/p.$$\\
Taking inverse limits with respect to corestriction then gives a map

$$\res_{\ell}: \ilim \Sel(E/K_n) \to \ilim E(K_{n,\ell})/p.$$\\
Dualizing this map and composing it with the above isomorphism gives a map

$$\uppsi_{\ell}: \dlim H^1(K_{n,\ell}, E)[p] \to (\ilim \Sel(E/K_n))^{\dual}.$$\\
Now assume that $E$ has ordinary reduction at $p$. In \cite{Matar1} we analyzed the images of the maps $\uppsi_{\ell}$ for an infinite set of primes $\ell$ to conclude that $\rank_{\overbar{\Lambda}}(\ilim \Sel(E/K_n)) \le 1$ (see \cite{Matar2} theorem 3.1).

Now consider the group $\ilim \Sel(E/K_{\infty})^{\Gamma_n}$ where the inverse limit is taken with respect to the norm maps. Then a direct consequence of Mazur's control theorem (under the assumptions in \cite{Matar1}) gives that the map induced by restriction

$$\Xi: \ilim \Sel(E/K_n) \to \ilim \Sel(E/K_{\infty})^{\Gamma_n}$$\\
is an injection with finite cokernel.

Also one can show (\cite{Matar2} proposition 2.1) that $\corank_{\overbar{\Lambda}}(\Sel(E/K_{\infty}))=\rank_{\overbar{\Lambda}}(\ilim \Sel(E/K_{\infty})^{\Gamma_n})$. So the results mentioned above show that $\corank_{\overbar{\Lambda}}(\Sel(E/K_{\infty})) \le 1$. But it is not hard to show that $\corank_{\overbar{\Lambda}}(\Sel(E/K_{\infty})) \geq 1$ and so we get $\corank_{\overbar{\Lambda}}(\Sel(E/K_{\infty}))=1$. By a little bit more work one can also show that $\corank_{\Lambda}(\Selinf(E/K_{\infty}))=1$. These two results are the content of theorem 3.4 of \cite{Matar2}.

In this article we reverse some of the maps above. Rather than take the direct limit of the groups $H^1(K_{n,\ell}, E)[p]$ with respect to restriction, we take their inverse limit with respect to corestriction $\ilim H^1(K_{n,\ell}, E)[p]$. Analogously, we take the direct limits of the groups $E(K_{n,\ell})/p$ with respect to restriction $\dlim E(K_{n,\ell})/p$. By local Tate duality we get an isomorphism

$$\ilim H^1(K_{n,\ell}, E)[p] \cong (\dlim E(K_{n,\ell})/p)^{\dual}.$$\\
As we explained above, we have a map

$$\res_{\ell}: \Sel(E/K_n) \to E(K_{n,\ell})/p.$$\\
Taking direct limits with respect to restriction then gives a map

$$\res_{\ell}: \dlim \Sel(E/K_n) \to \dlim E(K_{n,\ell})/p.$$\\
Recall that $\Sel(E/K_{\infty})=\dlim \Sel(E/K_n)$. Therefore by dualizing the map $\res_{\ell}$ and composing it with the above isomorphism, we obtain a map

$$\uppsi_{\ell}: \ilim H^1(K_{n,\ell}, E)[p] \to \Sel(E/K_{\infty})^{\dual}.$$\\
In a somewhat similar way to \cite{Matar1}, we study the images of the maps $\uppsi_{\ell}$ and show that as the primes $\ell$ range over an appropriate infinite set the images of $\uppsi_{\ell}$ generate $\Sel(E/K_{\infty})^{\dual}$. Under the hypothesis of our theorem above this will allow us to prove that $\Sel(E/K_{\infty})^{\dual}$ is a free $\overbar{\Lambda}$-module of rank 1. Together with a few observations we will obtain the results of the theorem. However we should note that our proof relies on the fact mentioned above that $\corank_{\overbar{\Lambda}}(\Sel(E/K_{\infty}))=1$. Therefore this article is a natural continuation of \cite{Matar1}.

\section{Preliminaries}

\subsection{Notation}

In this section we explain the notation we will use throughout the paper. We will use any notation and definitions listed in the introduction and in addition introduce more in this section.

We fix a complex conjugation $\tau$ on $\overbar{\Q}$ (the algebraic closure of $\Q$). Given a $\Z[\frac{1}{2}][\tau]$-module $M$, we have a decomposition $M=M^+ \oplus M^-$ where $M^+$ and $M^-$ denotes the submodule on which $\tau$ acts as $+1$, respectively $-1$. Also, if $x \in M$ and $X \subset M$, we let

$$x^{\pm}=\frac{1}{2}(x\pm \tau x)$$
$$X^{\pm}=\{x^{\pm} \; | \; x \in X\}$$\\
For any $m$ we let $K[m]$ denote the ring class field of $K$ of conductor $m$. Let $K[p^{\infty}]=\cup_{n \geq 1} K[p^n]$. Then  $\Gal(K[p^{\infty}]/K)$ is isomorphic to $\Zp \times \Delta$, where $\Delta$ is a finite abelian group. The unique $\Zp$-extension that is contained in $K[p^{\infty}]/K$ is the anticyclotomic $\Zp$-extension of $K$ which we will denote by $K_{\infty}/K$.

If $\ell$ is a rational prime and $F$ is a number field we define

\begin{flalign*}
\qquad \qquad &E(F_{\ell})/p:=\oplus_{\lambda|\ell}E(F_{\lambda})/p&\\[0.5em]
&H^1(F_{\ell}, E[p]):=\oplus_{\lambda|\ell}H^1(F_{\lambda}, E[p])&\\[0.5em]
&H^1(F_{\ell}, E)[p]:=\oplus_{\lambda|\ell}H^1(F_{\lambda}, E)[p],
\end{flalign*}\\
\noindent where the sum is taken over all primes of $F$ dividing $\ell$.

With this notation we let $\res_{\ell}$ be the localization map:

\begin{flalign*}
\qquad \qquad &\res_{\ell}: E(F)/p \to E(F_{\ell})/p&\\[0.5em]
&\res_{\ell}: H^1(F, E[p]) \to H^1(F_{\ell}, E[p])&\\[0.5em]
&\res_{\ell}: H^1(F, E)[p] \to H^1(F_{\ell}, E)[p].
\end{flalign*}\\
If $F=K_n$, with the above notation we let $K_{n,\ell}$ denote $F_{\ell}$.

We will frequently write $\dlim$ (resp. $\ilim$) for $\dlim_n$ (resp. $\ilim_n$) as our limits are taken over $n$.

\subsection{Heegner points and Kolyvagin classes}
We fix a modular parametrization $\pi: X_0(N) \to E$ which maps the cusp $\infty$ of $X_0(N)$ to the origin of $E$ (see \cite{Wiles} and \cite{BDCT})
If we assume that every prime dividing $N$ splits in $K/\Q$ (condition $(\star)$-i), then it follows that we can choose an ideal $\cN$ such that $\cO_K/\cN \cong \Z/N\Z$. Let $m$ be an integer that is relatively prime to $Nd_K$ and let $\cO_m = \Z + m\cO_K$ be the order of conductor $m$ in $K$. The ideal $\cN_m=\cN \cap \cO_m$ satisfies $\cO_m/\cN_m \cong \Z/N\Z$ and therefore the natural projection of complex tori:

$$\C/\cO_m \to \C/\cN_m^{-1}$$\\
is a cyclic $N$-isogeny, which corresponds to a point of $X_0(N)$. Let $\alpha[m]$ be its image under the modular parametrization $\pi$. From the theory of complex multiplication we have that $\alpha[m] \in E(K[m])$ where $K[m]$ is the ring class field of $K$ of conductor $m$.

If we assume that the class number of $K$ is not divisible by $p$ (condition $(\star)$-ii), it follows for any $n$ that $K[p^{n+1}]$ is the ring class field of minimal conductor that that contains $K_n$. We now define $\alpha_n \in E(K_n)$ to be the trace from $K[p^{n+1}]$ to $K_n$ of $\alpha[p^{n+1}]$.

Let $R_n\alpha_n$ denote the $R_n$-submodule of $H^1(K_n, E[p])$ generated by the image of $\alpha_n$ under the Kummer map

$$E(K_n) \to H^1(K_n, E[p]).$$\\

If we assume that $\Gal(\Q(E[p])/\Q)=GL_2(\Fp)$ (condition $(\star)$-iii), then by corollary 2.4 of \cite{Matar1} we have $E(K_{\infty})[p^{\infty}]=\{0\}$. This implies that the restriction map for $m\geq n$

$$H^1(K_n, E[p]) \to H^1(K_m, E[p])$$\\
\noindent is injective and therefore allows us to view $R_n\alpha_n$ as a submodule of $H^1(K_m, E[p])$.

Now assume $E$ has ordinary reduction at $p$ (condition $(\star)$-iv), $p \nmid \#E(\Fp)$ (condition $(\star)$-i) and conditions $(\star)$-v and $(\star)$-vi hold when $p$ is inert in $K$, respectively $p$ splits in $K$. From section 3.3 of \cite{PR} it follows that

\begin{align}
&\left\{
\begin{aligned}
\text{Tr}_{K_1/K}(\alpha_1) &= (a_p-a_p^{-1}(p+1))\alpha_0 \quad \text{if $p$ is inert in $K/\Q$}\\
\text{Tr}_{K_1/K}(\alpha_1) &= (a_p-(a_p-2)^{-1}(p-1))\alpha_0 \quad \text{if $p$ splits in $K/\Q$}\\
\end{aligned}\label{HP1}
\right.\\[10pt]
&\text{Tr}_{K_{n+1}/K_n}(\alpha_{n+1})=a_p\alpha_n-\alpha_{n-1} \quad \text{for $n \geq 1$} \label{HP2}
\end{align}\\

We claim that the map $\text{Tr}_{K_{n+1}/K_n}: R_{n+1} \alpha_{n+1} \to R_n \alpha_n$ is surjective. We argue by induction on $n$. The assumptions above imply that the numbers $a_p-a_p^{-1}(p+1)$ and $a_p-(a_p-2)^{-1}(p-1)$ are $p$-adic units when $p$ is inert in $K$, respectively $p$ splits in $K$. This proves the claim for $n=0$. Assume the claim for $n-1$. Then $\alpha_{n-1} = u \text{Tr}_{K_n/K_{n-1}}(\alpha_n)$ for some unit $u$ in $R_{n-1}$ Then from (\ref{HP2}) above we get that $\text{Tr}_{K_{n+1}/K_n}(\alpha_{n+1})=(a_p-u\text{Tr}_{K_n/K_{n-1}})\alpha_n$. Since $a_p \not\equiv 0 \mod p$, therefore $a_p-u\text{Tr}_{K_n/K_{n-1}}$ is a unit in $R_n$. The claim for $n$ follows.

As in \cite{Matar1}, we now describe the construction of Kolyvagin classes over ring class fields. For the rest of this section we assume that $(E,p)$ satisfies $(\star)$. First let us make the following definition

\begin{definition} A rational prime $\ell$ is called a \textit{Kolyvagin prime} if\\
(i) $\ell$ is relatively prime to $pNd_K$\\
(ii) $\Frob_{\ell}(K(E[p])/\Q)=[\tau]$
\end{definition}

Let $r$ be a squarefree product of Kolyvagin primes. For any $n$ let $K_n[r]$ denote the field $K_nK[r]$. We now define $\alpha_n(r)$ to be the trace of $\alpha[rp^{n+1}]$ from $K[rp^{n+1}]$ to $K_n[r]$.

Let $G_{n,r}=\Gal(K_n[r]/K_n[1])$ and let $G_{n,\ell}=\Gal(K_n[\ell]/K_n[1])$. By class field theory $G_{n,r}=\prod_{\ell|r}G_{n,\ell}$ and $G_{n,\ell}$ is cyclic of order $\ell+1$. Let $\sigma_{\ell}$ be a generator of $G_{n,\ell}$. Define $D_{\ell} :=\sum_{i=1}^{\ell} i\sigma_{\ell}^i \in \cy{p}[G_{n,\ell}]$ and $D_r:=\prod_{\ell|r}D_{\ell} \in \cy{p}[G_{n,r}]$. Then one can show that $D_r\alpha_n(r)$ belongs to $(E(K_n[r]/p))^{G_{n,r}}$ (see \cite{BD} lemma 3.3). It follows that $\text{Tr}_{K_n[1]/K_n} D_r \alpha_n(r) \in (E(K_n[r])/p)^{\cG_n,r}$ where $\cG_{n,r}=\Gal(K_n[r]/K_n)$. Now consider the commutative diagram

\begin{equation}\label{kolyvagin_classes_diagram}
\xymatrix{
&&& 0 \ar[d]\\
&&& H^1(K_n[r]/K_n, E)[p] \ar[d]_{\infl}\\
0 \ar[r] & E(K_n)/p \ar[d] \ar[r]^-{\phi} & H^1(K_n, E[p]) \ar[d]^{\rotatebox{90}{$\text{\~{}}$}}_{\res} \ar[r] & H^1(K_n, E)[p] \ar[d]_{\res} \ar[r] & 0\\
0 \ar[r] & (E(K_n[r])/p)^{\cG_{n,r}} \ar[r]^-{\phi_r} & H^1(K_n[r], E[p])^{\cG_{n,r}} \ar[r] & H^1(K_n[r],E)[p]^{\cG_{n,r}}
}
\end{equation}

Let $c_n(r) \in H^1(K_n, E[p])$ be so that

$$\phi_r(\text{Tr}_{K_n[1]/K_n} D_r\alpha_n(r))=\text{Res}(c_n(r))$$\\
and let $d_n(r)$ be the image of $c_n(r)$ in $H^1(K_n, E)[p]$. Note that $c_n(1) = \phi(\alpha_n)$.\\
These Kolyvagin classes have the following properties:
\begin{enumerate}
\item Let $-\epsilon$ denote the sign of the functional equation of the L-function of $E/\Q$, $f_r$ be the number of prime divisors of $r$. We have $\tau\alpha_n=\epsilon g^{i_{n,1}}\alpha_n + \beta_n$ with $\beta_n \in E(K_n)_{\text{tors}}$, $g$ a generator of $\Gal(K_{\infty}/K)$ and $i_{n,1} \in \{0,...,p^n-1\}$. Moreover, $\tau$ acts on $H^1(K_n, E[p])$ and we can deduce that $\tau c_n(r)=\epsilon_r g^{i_{n,r}}c_n(r)$ where $\epsilon_r=(-1)^{f_r}\epsilon$ and $i_{n,r}\in \{0,...,p^n-1\}$.
\item If $v$ is a rational prime that does not divide $r$, then $d_n(r)_{v_n}=0$ in $H^1(K_{v_n}, E)[p]$ for all primes of $K_n \, \, v_n|v$.
\item If $\ell|r$, there exists a $G_n$-equivariant and a $\tau$-antiequivariant isomorphism:
$$\psi_{n, \ell}: H^1(K_{n,\ell}, E)[p] \to E(K_{n,\ell})/p$$
such that $\psi_{n,\ell}(\res_{\ell}d_n(r))=\res_{\ell}(c_n(r/\ell))$.\\
If we let $\cor_{n+1}$ denote the corestriction maps $H^1(K_{n+1,\ell}, E)[p] \to H^1(K_{n, \ell}, E)[p]$ and $E(K_{n+1,\ell})/p \to E(K_{n, \ell})/p$, then we have \\$\psi_{n, \ell} \circ \cor_{n+1} = \cor_{n+1} \circ \psi_{n+1, \ell}$.
\item The following corestriction maps are surjective: \\$\cor_{n+1}: R_{n+1}\alpha_{n+1} \to R_n\alpha_n$, $\cor_{n+1}: R_{n+1}c_{n+1}(r) \to R_nc_n(r)$ and $\cor_{n+1}: R_{n+1}d_{n+1}(r) \to R_nd_n(r)$.\\
\end{enumerate}
The isomorphism $\psi_{n, \ell}$ is constructed in \cite{BD}. The fact that it commutes with corestriction is easily seen since any Kolyvagin prime $\ell$ splits completely in $K_{\infty}/K$.\\

Regarding the surjectivity of the corestriction maps in (4), the surjectivity  of $\cor_{n+1}: R_{n+1}\alpha_{n+1} \to R_n\alpha_n$ was proven earlier. By a similar proof using the  relations in \cite{PR} section 3.3, one can show that $\cor_{n+1}: R_{n+1}\alpha_{n+1}(r) \to R_n\alpha_n(r)$ is surjective (note that $\Gal(K_n[r]/K)=\Gal(K_n/K) \times \Gal(K[r]/K)$ and so $R_n \alpha_n(r)$ makes sense). It follows from the commutative diagram (\ref{kolyvagin_classes_diagram}) that $\cor_{n+1}: R_{n+1}c_{n+1}(r) \to R_nc_n(r)$ is surjective which in turn implies that $\cor_{n+1}: R_{n+1}d_{n+1}(r) \to R_nd_n(r)$ is also surjective.

\subsection{Preliminary Results}

In this section, we collect some preliminary results that will be used in the proof of our theorem in the introduction. We assume throughout this section that $\Gal(\Q(E[p])/\Q)=GL_2(\Fp)$ (condition $(\star)$-iii).

Now for any $n$ and any rational prime $\ell$, local Tate duality gives a non-degenerate pairing (see \cite{Gross} prop. 7.5)

\begin{equation}
\langle \; , \; \rangle_{\ell}: E(K_{n,\ell})/p \times H^1(K_{n,\ell}, E)[p] \to \Fp
\end{equation}\\
This identifies $H^1(K_{n, \ell}, E)[p]$ with $(E(K_{n,\ell})/p)^{\dual}$.

Moreover, if $a \in E(K_{n+1,\ell})/p$ and $b \in H^1(K_{n,\ell}, E)[p]$, then a property of Tate local duality gives $\langle\res a, b\rangle=\langle a, \cor b\rangle$ where

$$\cor: H^1(K_{n+1, \ell}, E)[p] \to H^1(K_{n, \ell}, E)[p]$$\\
is the corestriction map.

$$\res : E(K_{n, \ell})/p \to E(K_{n+1, \ell})/p$$\\
is the restriction map. Therefore Tate local duality induces an isomorphism

\begin{equation}\label{Tatelocalduality_isomorphism}
\ilim H^1(K_{n,\ell}, E)[p] \cong (\dlim E(K_{n,\ell})/p)^{\dual}
\end{equation}\\
where the inverse limit is taken over $n$ with respect to the corestriction maps and the direct limit is taken over $n$ with respect to the restriction maps.

The $p$-Selmer group $\Sel(E/K_n)$ consists of the cohomology classes $s \in H^1(K_n, E[p])$ whose restrictions $\res_v(s)\in H^1(K_{n,v}, E[p])$ belong to $E(K_{n,v})/p$ for all primes $v$ of $K_n$, where we view $E(K_{n,v})/p$ as a subspace of $H^1(K_{n,v}, E[p])$ using the Kummer sequence

$$\displaystyle 0 \longrightarrow E(K_{n,v})/p \longrightarrow H^1(K_{n,v}, E[p]) \longrightarrow H^1(K_{n,v}, E)[p] \longrightarrow 0.$$\\
Therefore, we have a map

\begin{equation}
\res_{\ell}: \Sel(E/K_n) \to E(K_{n,\ell})/p.
\end{equation}\\
The map $\res_{\ell}$ then induces a map

\begin{equation}
\res_{\ell}: \dlim \Sel(E/K_n) \to \ilim E(K_{n,\ell})/p.
\end{equation}\\
Note that $\Sel(E/K_{\infty})=\dlim \Sel(E/K_n)$. Therefore, dualizing the map $\res_{\ell}$ and using the isomorphism (\ref{Tatelocalduality_isomorphism}) above we get a map

$$\uppsi_{\ell}: \ilim H^1(K_{n,\ell}, E)[p] \to \Sel(E/K_{\infty})^{\dual}.$$\\
We will need to understand the $\overbar{\Lambda}$-structure of $\img \uppsi_{\ell}$. In order to do this, we will need the following important observation

\begin{proposition}\label{Iwasawa_rank_proposition}
If $\ell$ is a Kolyvagin prime, then $\ilim H^1(K_{n,\ell}, E)[p]$ is a free $\overbar{\Lambda}$-module of rank 2.
\end{proposition}
\begin{proof}
Consider the $\overbar{\Lambda}$-module $M:=\dlim H^1(K_{n, \ell}, E)[p]$ where the direct limit is taken with respect to the restriction maps. In \cite{Matar1} proposition 2.5 we showed that $M$ is a cofree $\overbar{\Lambda}$-module of rank 2 and that $M^{\Gamma_n}=H^1(K_{n, \ell}, E)[p]$ for any $n$. Therefore the proposition follows from \cite{Matar2} proposition 2.1.
\end{proof}

Now for any $n$ let $L_n=K_n(E[p])$ and $\cG_n=\Gal(L_n/K_n)$ which is isomorphic to $GL_2(\Fp)$ by \cite{Matar1} lemma 2.3. Then we have the following proposition (\cite{Gross} prop. 9.1)

\begin{proposition}\label{res_isom_prop}
The restriction map induces an isomorphism:
$$\res: H^1(K_n, E[p]) \isomarrow H^1(L_n, E[p])^{\cG_n} =\Hom_{\cG_n}(\Gal(\overbar{\Q}/L_n), E[p])$$
\end{proposition}

From the above proposition we get a pairing

\begin{equation}
[\; , \;]: H^1(K_n, E[p]) \times \Gal(\overbar{\Q}/L_n) \to E[p]
\end{equation}\\
If $S_n \subset H^1(K_n, E[p])$ is a finite subgroup, let $\Gal_{S_n}(\overbar{\Q}/L_n)$ be the subgroup consisting of $\rho \in \Gal(\overbar{\Q}/L_n)$ such that $[s,\rho]=0$ for all $s \in S_n$ and let $L_{S_n}$ be the fixed field of $\Gal_{S_n}(\overbar{\Q}/L_n)$. Then $L_{S_n}/K_n$ is a finite Galois extension and the above pairing induces a nondegenerate pairing

\begin{equation}\label{pairing}
[\; , \;]: S_n \times \Gal(L_{S_n}/L_n) \to E[p]
\end{equation}\\

We now assume that we have a finite subgroup $S \subset H^1(K, E[p])$ that is stable under $\Gal(K/\Q)$. Then $L_S/\Q$ is a finite Galois extension. Let $V=\Gal(L_S/L_0)$. Given a subset $U$ of $V$ we define

$$\mathscr{L}(U)=\{\ell \; \text{rational prime}\; | \;\ell \nmid pN\; \text{and} \; \Frob_{\ell}(L_S/\Q)=[\tau u] \; \text{for} \; u \in U \}$$

Note that every $\ell \in \mathscr{L}(U)$ is a Kolyvagin prime. A suitably chosen subgroup $S$ will play an important role in our proof of the theorem in the introduction

\begin{proposition}\label{generating_Selmer_proposition}
If $U^+$ generates $V^+$, then $\img \uppsi_{\ell}$ with $\ell$ ranging over $\mathscr{L}(U)$ generate $\Sel(E/K_{\infty})^{\dual}$
\end{proposition}
\begin{proof}
Let $s \in \dlim \Sel(E/K_n)=\Sel(E/K_{\infty})$. To prove the proposition, it suffices to show that $\res_{\ell}(s)=0$ for all $\ell \in \mathscr{L}(U)$ implies $s=0$. By \cite{Matar1} corollary 2.4, we have that $E(K_{\infty})[p^{\infty}]=\{0\}$. From this it follows that the maps in the direct limit $\dlim \Sel(E/K_n)$ are injections. Also for any $\ell \in \mathscr{L}(U)$, since $\ell$ is inert in $K/\Q$ and $\ell \neq p$, we get that $\ell$ splits completely in $K_{\infty}/K$ so the maps in the direct limit $\dlim E(K_{n,\ell})/p$ are also injections. These observations imply that to prove the proposition, we need to show that for any $n$ if $s \in \Sel(E/K_n)$ and $\res_{\ell}(s)=0$ for all $\ell \in \mathscr{L}(U)$, then $s=0$. This can be shown exactly as in \cite{Matar1} proposition 2.8.
\end{proof}

The following proposition will be an important tool to finding relations in $\Sel(E/K_{\infty})^{\dual}$

\begin{proposition}\label{global_duality_proposition}

For any $n$, if $s \in \Sel(E/K_n)$ and $\gamma \in H^1(K_n, E)[p]$, then

$$\sum_{\ell} \langle \res_{\ell} s, \; \res_{\ell} \gamma \rangle_{\ell} =0$$

where the sum is taken over all the rational primes

\end{proposition}

The proposition is an immediate consequence of the global reciprocity law for elements in the Brauer group of $K_n$ (\cite{NSW} th. 8.1.17), taking into account the definition of local Tate duality (\emph{loc. cit.} th. 7.2.6).

\section{Proof of Theorem}
In this section we prove the theorem in the introduction. We assume throughout this section the assumptions of the theorem, namely that $(E, p)$ satisfies $(\star)$ and that $p$ does not divide $y_K$ in $E(K)$. Recall that $y_K= \text{Tr}_{K_1/K}(\alpha[1])$ (in the introduction $\alpha[1]$ was also denoted $x_1$) and $\alpha_0= \text{Tr}_{K[p]/K}(\alpha[p])$. From the relations in \cite{PR} section 3.3, we see that $\alpha_0= a y_K$ where $a=a_p$ and $a=a_p-2$ when $p$ is inert in $K$, respectively $p$ splits in $K$. Our conditions (conditions $(\star)$-iv and $(\star)$-vi) imply that $a$ is not divisible by $p$. So since $p$ does not divide $y_K$ in $E(K)$, it also follows that $p$ does not divide $\alpha_0$ in $E(K)$. Let $\delta \alpha_0$ be the nonzero image of $\alpha_0$ in $H^1(K, E[p])$.

Consider the restriction maps $\res_n: H^1(K_n, E[p]) \to H^1(K_{n+1}, E[p])$. These maps are injective as was explained in section 2.2. Since the corestriction maps $\cor_{n+1}: R_{n+1} \alpha_{n+1} \to R_n \alpha_n$ are surjective by property (4) of the Kolyvagin classes in section 2.2, therefore it follows that $\res_n(R_n \alpha_n) \subseteq R_{n+1} \alpha_{n+1}$. So, as in \cite{Matar1}, we may for the direct limit $\dlim R_n \alpha_n$. Then we have the following theorem

\begin{theorem}\label{direct_limit_theorem}
As a $\overbar{\Lambda}$-module $(\dlim R_n \alpha_n)^{\dual}$ is finitely generated and not torsion.
\end{theorem}
\begin{proof}
It is well-known (see for example \cite{Manin} th. 4.5) that $\Selinf(E/K_{\infty})^{\dual}$ is a finitely generated $\Lambda$-module. Since $E(K_{\infty})[p^{\infty}]=\{0\}$ by \cite{Matar1} corollary 2.4, therefore we have an isomorphism $\Sel(E/K_{\infty}) \isomarrow \Selinf(E/K_{\infty})[p]$ and so $\Sel(E/K_{\infty})^{\dual}$ is a finitely generated $\overbar{\Lambda}$-module. The same then holds for $(\dlim R_n \alpha_n)^{\dual}$ (since it is a quotient of $\Sel(E/K_{\infty})^{\dual}$).

We now prove that $(\dlim R_n \alpha_n)^{\dual}$ is not $\overbar{\Lambda}$-torsion. Since finitely generated torsion $\overbar{\Lambda}$-modules are finite, we just have to show that $\dlim R_n \alpha_n$ has infinite cardinality. Clearly, this will hold if we show that the inverse limit $\ilim R_n \alpha_n$ defined using surjective corestriction maps (see property (4) of the Kolyvagin classes in section 2.2) is nonzero. This last fact holds since $\delta \alpha_0 \neq 0$.
\end{proof}

\begin{remark}
In both \cite{Matar1} and \cite{Matar2}, the condition that $p$ does not divide $\varphi(Nd_K)$ nor the number of geometrically connected components of the kernel of $\pi^*: J_0(N) \to E$ appears. Let us call this condition (C). When $E$ has ordinary reduction at $p$ (which is the case we are considering in this paper), condition (C) is used via the work of Cornut \cite{Cornut} to prove \cite{Matar1} theorem 3.1 which has an identical statement to the above theorem. As we just saw, the above theorem holds without the need for condition (C). For $\delta \alpha_0 \neq 0$ ensures that the above theorem is true. Therefore, we see that theorem A in \cite{Matar1} and theorem 3.4(a) in \cite{Matar2} (whose proof is based on theorem A of \cite{Matar1}) are true if $(E,p)$ satisfies ($\star$) and $\delta \alpha_0 \neq 0$.
\end{remark}

We now define $X_{s,p}(E/K_{\infty}):=\ilim \Sel(E/K_n)$ where the inverse limit is taken over $n$ with respect to the corestriction maps. Note that we have chosen to put an ``s" in the subscript so that the reader does not confuse this group with the group $X_p(E/K_{\infty})$ in \cite{Matar1} which was defined in a different way.

We also define $Y_{s,p}(E/K_{\infty})=\ilim \Sel(E/K_{\infty})^{\Gamma_n}$ where the inverse limit is taken over $n$ with respect to the norm maps.

We also consider the group $\ilim R_n \alpha_n$ where the inverse limit is taken with respect to surjective corestriction maps (see property (4) of the Kolyvagin classes in section 2.2).

For each $n$ we have an injection $\iota_n: R_n \alpha_n \hookrightarrow \Sel(E/K_n)$. These injections induce an injection

$$\iota: \ilim R_n \alpha_n \hookrightarrow X_{s,p}(E/K_{\infty}).$$\\
Also, for each $n$ we have a restriction map $\res_n: \Sel(E/K_n) \to \Sel(E/K_{\infty})^{\Gamma_n}$. These restriction maps induce a map

$$\Xi: X_{s,p}(E/K_{\infty}) \to Y_{s,p}(E/K_{\infty}).$$\\
We now have the following important theorem

\begin{theorem}\label{Heegner_module_rank_theorem}
The group $Y_{s,p}(E/K_{\infty})$ is a free $\overbar{\Lambda}$-module of rank 1 and the maps $\iota$ and $\Xi$ are isomorphisms (hence $X_{s,p}(E/K_{\infty})$ and $\ilim R_n \alpha_n$ are also free $\overbar{\Lambda}$-modules of rank 1).
\end{theorem}

\begin{proof}
First note that by corollary 2.4 in \cite{Matar1}, we have that $E(K_{\infty})[p^{\infty}]=\{0\}$. Then since $E$ has good ordinary reduction at $p$, $p \nmid \#E(\Fp) \cdot \prod_{v | N} c_v$ and $E(K_{\infty})[p^{\infty}]=\{0\}$, therefore by Mazur's control theorem (see \cite{Mazur} and \cite{Gb_IIC}) the maps $\res_n: \Selinf(E/K_n) \to \Selinf(E/K_{\infty})^{\Gamma_n}$ are isomorphisms for all $n$.

Since $E(K_{\infty})[p^{\infty}]=\{0\}$, therefore we have an isomorphism $\Sel(E/K_{\infty}) \isomarrow \Selinf(E/K_{\infty})[p]$. So Mazur's control theorem implies that the restriction maps $\res_n: \Sel(E/K_n) \to \Sel(E/K_{\infty})^{\Gamma_n}$ are isomorphisms. This implies that the map $\Xi$ is an isomorphism.

By \cite{Matar2} theorem 3.4 (and the remark after theorem \ref{direct_limit_theorem}), $\Selinf(E/K_{\infty})^{\dual}$ has $\Lambda$-rank equal to 1 and $\mu$-invariant equal to zero. Also as above we have $\Sel(E/K_{\infty}) \isomarrow \Selinf(E/K_{\infty})[p]$.  These 2 facts imply that $\Sel(E/K_{\infty})^{\dual}$ has $\overbar{\Lambda}$-rank equal to 1. This then implies by \cite{Matar2} proposition 2.1 that $Y_{s,p}(E/K_{\infty})$ is a free $\overbar{\Lambda}$-module of rank 1. Since $\Xi$ is an isomorphism, therefore $X_{s,p}(E/K_{\infty})$ is also a free $\overbar{\Lambda}$-module of rank 1.

Since $\iota$ is an injection and $X_{s,p}(E/K_{\infty})$ is a free $\overbar{\Lambda}$-module of rank 1, therefore to prove that $\iota$ is an isomorphism we only have to show that $\img \iota \notin (\gamma-1)X_{s,p}$. We show this as follows. First, for the inverse limits $\ilim R_n \alpha_n$ and $X_{s,p}(E/K_{\infty})$, let $\pi_0$ be the projection onto its zeroth component:

$$\pi_0: \ilim R_n \alpha_n \to \Fp \alpha_0$$
$$\pi_0: X_{s,p}(E/K_{\infty}) \to \Sel(E/K)$$\\
Since maps defining the inverse limit $\ilim R_n \alpha_n$ are surjective, therefore there exists $a \in \ilim R_n \alpha_n$ such that $\pi_0(a)= \delta \alpha_0$. Now consider any element $b \in (\gamma-1)X_{s,p}(E/K_{\infty})$. Since $\Sel(E/K)$ is invariant under $\Gamma$, therefore it follows that $\pi_0(b)=0$. But we have $\pi_0(\iota(a))=\iota_0(\delta \alpha_0)$ which is not zero since $\iota_0$ is an injection and $\delta \alpha_0$ is not zero. This shows that $\iota(a) \notin (\gamma-1)X_{s,p}(E/K_{\infty})$ which, as explained above, completes the proof.
\end{proof}

By property (1) of the Kolyvagin classes in section 2.2, taking into account the fact that $E(K_{\infty})[p^{\infty}]=\{0\}$ (\cite{Matar1} corollary 2.4), we have that $\tau \, \delta \alpha_0 = \epsilon \, \delta \alpha_0$ where $-\epsilon$ is the sign of the functional equation of the L-function of $E/\Q$. let $T$ be the subgroup generated by $\delta \alpha_0$ in $H^1(K, E[p])$. With the notation following proposition \ref{res_isom_prop}, we have an extension $L_T/\Q$ which is Galois over $\Q$ since $T$ is $\tau$-invariant. Now let $H=\Gal(L_T/L) \cong E[p]$ (see \cite{Gross} prop. 9.3) and choose $h \in H$ such that $(\tau h)^2 \in H^+ - \{0\}$. We now choose an auxiliary prime $\ell_1$ such that $\ell_1$ is relatively prime to $pNd_K$ and $\Frob_{\ell_1}(L_T/\Q)=[\tau h]$ (such a prime exists by the Chebotarev density theorem).

We now claim that $\res_{\ell_1} \delta \alpha_0 \neq 0$. To prove this we only have to note that $\ell_1$ is inert in $K/\Q$ and hence $\Frob_{\ell_1}(L_T/K)=[(\tau h)^2]$. Since $(\tau h)^2$ is nonzero, the fact that $\res_{\ell_1} \delta \alpha_0 \neq 0$ follows easily from the non-degeneracy of the pairing (\ref{pairing}).

By property (3) of the Kolyvagin classes in section 2.2, it follows that $\res_{\ell_1} d_0(\ell_1) \neq 0$. This implies that $\res_{\ell_1} c_0(\ell_1) \neq 0$ which in turn implies that $c_0(\ell_1) \neq 0$. By property (1) of the Kolyvagin classes in section 2.2, we have $\tau c_0(\ell_1)= -\epsilon \, c_0(\ell_1)$.

Similarly to theorem \ref{Heegner_module_rank_theorem}, we have the following

\begin{theorem}\label{Kolyvagin_classes_rank_theorem}
$\ilim R_n c_n(\ell_1)$ is a free $\overbar{\Lambda}$-module of rank 1.
\end{theorem}

\begin{proof}

If $L/\Q$ is an algebraic extension and $T$ is any set of primes of $L$ we define $\Selm_p^T(E/L)$ by the exact sequence

$$\displaystyle 0 \longrightarrow \Selm_p^T(E/L) \longrightarrow H^1(L, E[p])\longrightarrow \prod_{v \notin T} H^1(L_v, E)[p]$$

For any $n$, let $T_n$ be the primes of $K_n$ above $\ell_1$ and $T_{\infty}$ be the primes of $K_{\infty}$ above $\ell_1$. To prove the theorem, it suffices to show that the group $\ilim \Selm_p^{T_n}(E/K_n)$ (inverse limit with respect to corestriction) has $\overbar{\Lambda}$-rank less than or equal to one. To see why this suffices, note that since $E(K_{\infty})[p^{\infty}]=\{0\}$ (\cite{Matar1} corollary 2.4), $\ilim \Selm_p^{T_n}(E/K_n)$ injects into $\ilim \Selm_p^{T_{\infty}}(E/K_{\infty})^{\Gamma_n}$ (inverse limit with respect to norms). But by the argument in \cite{Matar1} proposition 3.2, $\Selm_p^{T_{\infty}}(E/K_{\infty})$ is a finitely generated $\overbar{\Lambda}$-module. This implies by \cite{Matar2} proposition 2.1 that $\ilim \Selm_p^{T_{\infty}}(E/K_{\infty})$ is a free $\overbar{\Lambda}$-module and since $\ilim \Selm_p^{T_n}(E/K_n)$ injects into it, it is also a free $\overbar{\Lambda}$-module (recall that $\overbar{\Lambda}$ is a PID).

Note that by property (2) of the Kolyvagin classes in section 2.2, $R_n c_n(\ell_1)$ is contained in $\Selm_p^{T_n}(E/K_n)$ and hence $\ilim R_n c_n(\ell_1) \subseteq \ilim \Selm_p^{T_n}(E/K_n)$. Also note that $\ilim R_n c_n(\ell_1)$ is not zero because $c_0(\ell_1)$ is not zero and the corestriction maps defining the inverse limit are surjective (see property (4) of the Kolyvagin classes in section 2.2). Now assume that $\rank_{\overbar{\Lambda}}(\ilim \Selm_p^{T_n}(E/K_n)) \le 1$. Then as $\ilim \Selm_p^{T_n}(E/K_n)$ is a free $\overbar{\Lambda}$-module and $\ilim R_n c_n(\ell_1)$ is a nonzero submodule of it, we must have that $\ilim R_n c_n(\ell_1)$ is a free $\overbar{\Lambda}$-module of rank 1.

So we see that it suffices to show that $\rank_{\overbar{\Lambda}}(\ilim \Selm_p^{T_n}(E/K_n)) \le 1$. We show this by adapting the proof of theorem A in \cite{Matar1} to our setup.

Let $\ell$ be a Kolyvagin prime different from $\ell_1$. In a similar way to \cite{Matar1} section 2.3, we will construct a map $\uppsi_{\ell}': \dlim H^1(K_{n,\ell}, E)[p] \to (\ilim \Selm_p^{T_n}(E/K_n))^{\dual}$ as follows:
First by Tate local duality we have an isomorphism

$$\dlim H^1(K_{n,\ell}, E)[p] \cong (\ilim E(K_{n,\ell})/p)^{\dual}.$$\\
Next, for any $n$, we have a restriction map $\res_{\ell}: \Selm_p^{T_n}(E/K_n) \to E(K_{n, \ell})/p$. Taking
inverse limits gives a map

$$\res_{\ell}: \ilim \Selm_p^{T_n}(E/K_n) \to \ilim E(K_{n, \ell})/p.$$\\
Dualizing this map and using the Tate duality isomorphism we get a map

$$\uppsi_{\ell}': \dlim H^1(K_{n, \ell}, E)[p] \to (\ilim \Selm_p^{T_n}(E/K_n))^{\dual}.$$\\
The map $\uppsi_{\ell}$ in \cite{Matar1} section 2.3 is similar to $\uppsi_{\ell}'$ but with a different codomain. We will modify the definition of $\mathscr{L}(U)$ on page 416 of \cite{Matar1} to exclude the prime $\ell_1$:

$$\mathscr{L}(U)=\{\ell \; \text{rational prime}\; | \;\ell \nmid p\ell_1N\; \text{and} \; \Frob_{\ell}(L_{S_{n_0}}/\Q)=[\tau u] \; \text{for} \; u \in U \}$$\\
Taking into account the remark after theorem \ref{direct_limit_theorem}, if we work with the map $\uppsi_{\ell}'$ rather than $\uppsi_{\ell}$, we obtain results identical those in \cite{Matar1} where the group $X_p(E/K_{\infty})$ gets replaced by $\ilim \Selm_p^{T_n}(E/K_n)$ so proposition 3.7 gives that $\rank_{\overbar{\Lambda}}(\ilim \Selm_p^{T_n}(E/K_n)) \le 1$ as desired. This proves the theorem.
\end{proof}

We now define the subgroup $S \subset H^1(K, E[p])$ and the set $U$ in section 2.3. Let $s=\delta \alpha_0$ and $s'=c_0(\ell_1)$. Now define $S$ to be the subgroup of $H^1(K, E[p])$ generated by $s$ and $s'$. Since $\tau s = \epsilon s$ and $\tau s' = -\epsilon s'$, the set $S$ is $\tau$-invariant and $\dim_{\Fp}(S)=2$. Let $V=\Gal(L_S/L)$ where $L=K(E[p])$. We will denote $L_{\{\Fp s\}}$ and $L_{\{\Fp s'\}}$ by $L_s$  and $L_{s'}$ respectively. By \cite{Gross} prop. 9.3 we have

$$V=\Gal(L_s/L) \times \Gal(L_{s'}/L) =E[p] \times E[p].$$\\
Complex conjugation $\tau$ acts on $V$ by

$$\tau(x,y)\tau = (\epsilon \, \tau x, -\epsilon \, \tau y).$$\\
Let $E[p]^{\epsilon}$ denote the submodule of $E[p]$ on which $\tau$ acts as $\epsilon$. We now define a subset $U$ of $V$ as

$$U=\{(x,y) \; | \; x \in E[p]^{\epsilon}-\{0\} \ \text{and} \ y \in E[p]^{-\epsilon} - \{0\}\}.$$\\
It is clear that $U^+$ generates $V^+$.\\

Before stating the next proposition note that if $M$ is a $\overbar{\Lambda}$-module on which $\tau$ acts, then $\tau$ acts on its coinvaraints $M_{\Gamma} = M/(\gamma-1)M$ as well. To see this, we need to show that $\tau (\gamma -1)M = (\gamma^{-1}-1)M$ is contained in $(\gamma-1)M$. Under the correspondence $T=\gamma-1$, this amounts to showing that the power series $(T+1)^{-1}-1$ is divisible by $T$ which is certainly true. In particular, if $M$ is a free $\overbar{\Lambda}$-module of rank 1, then $M_{\Gamma}$ is a 1-dimensional $\Fp$-vector space on which $\tau$ acts as $\omega \in \{+1, -1\}$

\begin{proposition}\label{Iwasawa_ranks_proposition}
For any $\ell \in \mathscr{L}(U)$, the submodules $\ilim \res_{\ell} R_n \alpha_n$ and $\ilim \res_{\ell} R_n c_n(\ell_1)$ of $\ilim E(K_{n,\ell})/p$ are free $\overbar{\Lambda}$-modules of rank 1 and together they generate $\ilim E(K_{n, \ell})/p$.
\end{proposition}
\begin{proof}
First let us note that $\res_{\ell} \ilim R_n \alpha_n = \ilim \res_{\ell} R_n \alpha_n$. This holds since for each $n$ the restriction map induces surjections $\phi_n: R_n \alpha_n \twoheadrightarrow \res_{\ell} R_n \alpha_n$ where the groups in the maps are finite and hence compact Hausdorff. So by a well-known result, taking inverse limits of the maps $\phi_n$ induces a surjection. Similarly, we have $\res_{\ell} \ilim R_n c_n(\ell) = \ilim \res_{\ell} R_n c_n(\ell_1)$. These 2 facts will be used throughout the proof.

For the inverse limits $\ilim R_n \alpha_n$, $\ilim R_n c_n(\ell_1)$, $\ilim E(K_{n, \ell})/p$, let $\pi_0$ be the projection onto its zeroth component:

$$\pi_0: \ilim R_n \alpha_n \to \Fp \alpha_0$$
$$\pi_0: \ilim R_n c_n(\ell_1) \to \Fp c_0(\ell_1)$$
$$\pi_0: \ilim E(K_{n, \ell})/p \to E(K_{\ell})/p$$\\
Recall that $\delta \alpha_0$ and $c_0(\ell_1)$ are nonzero and satisfy $\tau \, \delta \alpha_0 = \epsilon \, \delta \alpha_0$ and $\tau c_0(\ell_1) = -\epsilon \, c_0(\ell_1)$ where $-\epsilon$ is the sign of the L-function of $E/\Q$. Also recall that the corestriction maps defining the inverse limits $\ilim R_n \alpha_n$ and $\ilim R_n c_n(\ell_1)$ are surjective. It follows that there exist $a \in \ilim R_n \alpha_n$ and $b \in \ilim R_n c_n(\ell_1)$ such that $\pi_0(a)=\delta \alpha_0$ and $\pi_0(b)=c_0(\ell_1)$.

We claim that $\tau$ acts on $(\ilim R_n \alpha_n)_{\Gamma}$ as $\epsilon$ and on $(\ilim R_n c_n(\ell_1))_{\Gamma}$ as $-\epsilon$. To see this, assume that $\tau$ acts on $(\ilim R_n \alpha_n)_{\Gamma}$ as $\omega \in \{+1, -1\}$. Then $\tau a = \omega a +a'$ where $a' \in (\gamma-1)\ilim R_n \alpha_n$. Applying $\pi_0$ to both sides of this equality, we get $\tau \delta \alpha_0 = \omega \delta \alpha_0 + \pi_0(a')$. We have that $a' = (\gamma-1)a''$ for some $a'' \in \ilim R_n \alpha_n$. Since $\pi_0(a'')$ is $\Gamma$-invariant, therefore $\pi_0(a')=0$. This implies that $\omega=\epsilon$ as desired. Similarly, using the element $b$, one shows that $\tau$ acts on $(\ilim R_n c_n(\ell_1))_{\Gamma}$ as $-\epsilon$.

Now let $\ell \in \mathscr{L}(U)$. The maps in property (3) of the Kolyvagin classes in section 2.2 induce an isomorphism of $\overbar{\Lambda}$-modules $\ilim E(K_{n,\ell})/p \cong \ilim H^1(K_{n, \ell}, E)[p]$. This implies by proposition \ref{Iwasawa_rank_proposition} that $\ilim E(K_{n, \ell})/p$ is a free $\overbar{\Lambda}$-module of rank 2.

Our definition of the set $U$ ensures that $\res_{\ell}( \delta \alpha_0) \neq 0$ and $\res_{\ell}(c_0(\ell_1))\neq 0$. Therefore, $\pi_0(\res_{\ell}(a)) \neq 0$ and $\pi_0(\res_{\ell}(b))\neq 0$. In particular, $\ilim \res_{\ell} R_n \alpha_n \neq 0$ and $\ilim \res_{\ell} R_n c_n(\ell_1) \neq 0$. According to theorems \ref{Heegner_module_rank_theorem} and \ref{Kolyvagin_classes_rank_theorem}, both $\ilim R_n \alpha_n$ and $\ilim R_n c_n(\ell_1)$ are free $\overbar{\Lambda}$-modules of rank 1. Also since both $\ilim \res_{\ell} R_n \alpha_n$ and $\ilim \res_{\ell} R_n c_n(\ell_1)$ are contained in $\ilim E(K_{n, \ell})/p$ which as we mentioned above is a free $\overbar{\Lambda}$-module, therefore $\ilim \res_{\ell} R_n \alpha_n$ and $\ilim \res_{\ell} R_n c_n(\ell_1)$ are torsion-free. The above facts imply that $\ilim \res_{\ell} R_n \alpha_n$ and $\ilim \res_{\ell} R_n c_n(\ell_1)$ are free $\overbar{\Lambda}$-modules of rank 1.

Now consider any element $c \in (\gamma-1)\ilim E(K_{n,\ell})/p$. Since $E(K_{\ell})/p$ is $\Gamma$-invariant, therefore it follows that $\pi_0(c)=0$. This shows that $\res_{\ell}(a),\res_{\ell}(b) \notin (\gamma-1)\ilim E(K_{n, \ell})/p$ which in turn implies that the maps
$$\psi_1:(\ilim R_n \alpha_n)_{\Gamma} \to (\ilim E(K_{n,\ell})/p)_{\Gamma}$$
$$\psi_2: (\ilim R_n c_n(\ell_1))_{\Gamma} \to (\ilim E(K_{n,\ell})/p)_{\Gamma}$$\\
are nonzero. Since both $\ilim R_n \alpha_n$ and $\ilim R_n c_n(\ell_1)$ are free $\overbar{\Lambda}$-modules of rank 1, therefore $\dim_{\Fp}((\ilim R_n \alpha_n)_{\Gamma})=\dim_{\Fp}((\ilim R_n c_n(\ell_1))_{\Gamma})=1$. It follows that the maps $\psi_1$ and $\psi_2$ are injective. Also $\dim_{\Fp}((\ilim E(K_{n, \ell})/p)_{\Gamma})=2$ (because $\ilim E(K_{n,\ell})/p$ is a free $\overbar{\Lambda}$-module of rank 2) and, as we mentioned above, $\tau$ acts on $(\ilim R_n \alpha_ns)_{\Gamma}$ and $(\ilim R_n c_n(\ell_1))_{\Gamma}$ as $\epsilon$ and $-\epsilon$, respectively. Therefore, the images of the maps $\psi_1$ and $\psi_2$ generate $(\ilim E(K_{n, \ell})/p)_{\Gamma}$. Equivalently
$$\ilim  E(K_{n, \ell})/p = (\gamma-1) \ilim E(K_{n, \ell})/p + \ilim \res_{\ell} R_n \alpha_n + \ilim \res_{\ell} R_n c_n(\ell_1)$$
Since $\overbar{\Lambda}$ is a DVR with maximal ideal generated by $\gamma-1$ (it is isomorphic to the power series ring $\Fp[[T]]$ mapping $\gamma-1$ to $T$), therefore by Nakayama's lemma the above equality implies that $\ilim  E(K_{n, \ell})/p= \ilim \res_{\ell} R_n \alpha_n + \ilim \res_{\ell} R_n c_n(\ell_1)$. This completes the proof of the proposition.
\end{proof}

\begin{corollary}\label{Iwasawa_ranks_corollary}
For any $\ell \in \mathscr{L}(U)$, the submodules $\ilim \res_{\ell} R_n d_n(\ell)$ and $\ilim \res_{\ell} R_n d_n(\ell \ell_1)$ of $\ilim H^1(K_{n, \ell}, E)[p]$ are free $\overbar{\Lambda}$-modules of rank 1 and together they generate $\ilim H^1(K_{n, \ell}, E)[p]$.
\end{corollary}
\begin{proof}
Use the previous proposition and property (3) of the Kolyvagin classes in section 2.2.
\end{proof}

\begin{proposition}
$\img \uppsi_{\ell_1}$ is either finite or a free $\overbar{\Lambda}$-module of rank 1.
\end{proposition}
\begin{proof}
Note that $\ell_1$ was chosen so that $\res_{\ell_1}(\delta \alpha) \neq 0$. Therefore, just as in the proof of proposition \ref{Iwasawa_ranks_proposition} (see paragraphs 6 and 7 of the proof), $\ilim \res_{\ell_1} R_n \alpha_n$ is a free $\overbar{\Lambda}$-submodule of $\ilim E(K_{n, \ell_1})/p$ of rank 1 that is not contained in $(\gamma-1) \ilim E(K_{n, \ell_1})/p$. So using property (3) of the Kolyvagin classes in section 2.2, we get that $\ilim \res_{\ell_1} R_n d_n(\ell_1)$ is a free $\overbar{\Lambda}$-submodule of $\ilim H^1(K_{n, \ell_1}, E)[p]$ of rank 1 that is not contained in $(\gamma-1) \ilim H^1(K_{n, \ell_1}, E)[p]$.

By proposition \ref{global_duality_proposition} and property (2) of the Kolyvagin classes in section 2.2, $\uppsi_{\ell_1}(\ilim \res_{\ell_1} R_n d_n(\ell_1))=0$ so $\img \uppsi_{\ell_1}$ is the homomorphic image of $\ilim H^1(K_{n, \ell_1}, E)[p]/ \ilim \res_{\ell_1} R_n d_n(\ell_1)$. Since by proposition \ref{Iwasawa_rank_proposition}, $\ilim H^1(K_{n, \ell_1}, E)[p]$ is a free $\overbar{\Lambda}$-module of rank 2 and since $\ilim \res_{\ell_1} R_n d_n(\ell_1)$ is a free $\overbar{\Lambda}$-submodule of rank 1 that is not contained in $(\gamma-1)\ilim H^1(K_{n, \ell_1}, E)[p]$ therefore the quotient $\ilim H^1(K_{n, \ell_1}, E)[p] / \ilim \res_{\ell_1} R_n d_n(\ell_1)$ is a free $\overbar{\Lambda}$-module of rank 1. The result follows.
\end{proof}

The following is the essential step to proving our desired theorem

\begin{theorem}\label{Selmer_structure_theorem}
$\Sel(E/K_{\infty})^{\dual}$ is a free $\overbar{\Lambda}$-module of rank 1.
\end{theorem}
\begin{proof}
Let $\ell \in \mathscr{L}(U)$. Using proposition \ref{global_duality_proposition} together with property (2) of the Kolyvagin classes in section 2.2, we see that $\uppsi_{\ell}(\ilim \res_{\ell} R_n d_n(\ell))=0$. Therefore from corollary \ref{Iwasawa_ranks_corollary}, $\img \uppsi_{\ell} = \uppsi_{\ell}(\ilim \res_{\ell} R_n d_n(\ell \ell_1))$.

Now let $\alpha \in \ilim R_n d_n(\ell \ell_1)$. Then using proposition \ref{global_duality_proposition} together with property (2) of the Kolyvagin classes in section 2.2 it follows that

$$\uppsi_{\ell}(\res_{\ell}(\alpha)) + \uppsi_{\ell_1}(\res_{\ell_1}(\alpha)) = 0$$\\
Therefore, $\uppsi_{\ell}(\res_{\ell}(\alpha)) \in \img \uppsi_{\ell_1}$ and so by the above observation $\img \uppsi_{\ell} \subseteq \img \uppsi_{\ell_1}$ (note that $\res_{\ell} \ilim R_n d_n(\ell) = \ilim \res_{\ell} R_n d_n(\ell)$ by the argument in the beginning of the proof of proposition \ref{Iwasawa_ranks_proposition}).

Let $X$ be the submodule of $\Sel(E/K_{\infty})^{\dual}$ generated by $\img \uppsi_{\ell}$ as $\ell$ ranges over $\mathscr{L}(U)$. Then by our observation above, $X \subseteq \img \uppsi_{\ell_1}$. Also by proposition \ref{generating_Selmer_proposition}, $X=\Sel(E/K_{\infty})^{\dual}$. Since also $\uppsi_{\ell_1}$ is contained in $\Sel(E/K_{\infty})^{\dual}$, it therefore follows that $\img \uppsi_{\ell_1} = \Sel(E/K_{\infty})^{\dual}$.

Then by the previous proposition, it follows that $\Sel(E/K_{\infty})^{\dual}$ is either finite or a free $\overbar{\Lambda}$-module of rank 1. By \cite{Matar2} theorem 3.4 (and the remark after theorem \ref{direct_limit_theorem}), $\Selinf(E/K_{\infty})^{\dual}$ has $\Lambda$-rank equal to 1 and $\mu$-invariant equal to zero. Also according to corollary 2.4 in \cite{Matar1}, $E(K_{\infty})[p^{\infty}]=\{0\}$ which implies that we have an isomorphism $\Sel(E/K_{\infty}) \isomarrow \Selinf(E/K_{\infty})[p]$. These 2 facts imply that $\Sel(E/K_{\infty})^{\dual}$ has $\overbar{\Lambda}$-rank is equal to 1 so $\Sel(E/K_{\infty})^{\dual}$ cannot be finite. It is therefore a free $\overbar{\Lambda}$-module of rank 1.
\end{proof}

We need one more observation before proving our theorem

\begin{proposition}
$\rank (E(K_n)) \geq p^n$ for all $n \geq 0$
\end{proposition}
\begin{proof}
Consider the group $Y_{s,p}(E/K_{\infty})=\ilim \Sel(E/K_{\infty})^{\Gamma_n}$ where the inverse limit is with respect to norm maps. We claim that the transition norm maps are surjective. Let $X=\Sel(E/K_{\infty})^{\dual}$. By duality, to show that the transition norm maps defining $Y_{s,p}(E/K_{\infty})$ are surjective, it would suffice to show that for any $n$, the map $\theta_n : X_{\Gamma_n} \to X_{\Gamma_{n+1}}$ induced by $\text{Tr}_{K_{n+1}/K_n}$ is injective. Note that $X_{\Gamma_n}=X/(\gamma^{p^n}-1)X$, $X_{\Gamma_{n+1}}=X/(\gamma^{p^{n+1}}-1)X$ and $\text{Tr}_{K_{n+1}/K_n}=\frac{\gamma^{p^{n+1}}-1}{\gamma^{p^n}-1}=\sum_{i=0}^{p-1} \gamma^{ip^n}$.

Let $x \in X$ and assume that $\text{Tr}_{K_{n+1}/K_n}x = (\gamma^{p^{n+1}}-1)X$. Then $\text{Tr}_{K_{n+1}/K_n}x=(\gamma^{p^{n+1}}-1)y$ for some $y \in X$. This implies that $\text{Tr}_{K_{n+1}/K_n}(x-(\gamma^{p^n}-1)y)=0$. But by theorem \ref{Selmer_structure_theorem}, $X$ is a free $\overbar{\Lambda}$-module and so in particular is torsion-free. Therefore we have $x=(\gamma^{p^n}-1)y$. This shows that $\theta_n$ is injective and hence the transition maps defining $Y_{s,p}(E/K_{\infty})$ are surjective.

For the groups $\ilim R_n \alpha_n$ and $Y_{s,p}(E/K_{\infty})$, we let $\pi_n$ be the projection onto its $n$-th component.

$$\pi_n: \ilim R_k \alpha_k \to R_n \alpha_n$$
$$\pi_n: Y_{s,p}(E/K_{\infty}) \to \Sel(E/K_n)$$\\
Recall from the beginning of this section, that we have maps $\Xi \circ \iota: \ilim R_n \alpha_n \to Y_{s,p}(E/K_{\infty})$ and maps $\res_n \circ \; \iota_n: R_n \alpha_n \to \Sel(E/K_{\infty})^{\Gamma_n}$. These maps give a commutative diagram

\[
\xymatrixcolsep{5pc}
\xymatrix {
\ilim R_k \alpha_k \ar@{->>}[d]^{\pi_n} \ar@{->>}[r]^{\Xi \; \circ \; \iota} & Y_{s,p}(E/K_{\infty}) \ar@{->>}[d]^{\pi_n} \\
R_n \alpha_n \ar[r]^{\res_n \circ \; \iota_n} & \Sel(E/K_{\infty})^{\Gamma_n} }
\]\\

The surjectivity of the top row follows from theorem \ref{Heegner_module_rank_theorem}. As for the surjectivity of the vertical maps, this follows from the fact that the transition maps defining both $\ilim R_n \alpha_n$ and $Y_{s,p}(E/K_{\infty})$ are surjective (For $Y_{s,p}(E/K_{\infty})$, this is shown above. For $\ilim R_n \alpha_n$, see property (4) of the Kolyvagin classes in section 2.2).

We claim that the bottom map is an isomorphism. The injectivity follows from the fact that $\iota_n$ is injective and the fact that $E(K_{\infty})[p^{\infty}]=\{0\}$ (\cite{Matar1} corollary 2.4) which implies that $\res_n$ is injective. The surjectivity follows from the commutative diagram.

The conclusion of the proof now is easy. Let $n \geq 0$ be an integer. By theorem \ref{Selmer_structure_theorem}, $\Sel(E/K_{\infty})^{\dual}$ is a free $\overbar{\Lambda}$-module of rank 1. Therefore $\dim_{\Fp}(\Sel(E/K_{\infty})^{\Gamma_n})=p^n$ and so by the isomorphism of the bottom row of the commutative diagram, we have $\dim_{\Fp}(R_n \alpha_n)=p^n$. But $R_n \alpha_n$ is contained in the image of $E(K_n)/p \hookrightarrow H^1(K_n, E[p])$ via the injective Kummer map. Therefore, $\dim_{\Fp}(E(K_n)/p) \geq p^n$. Since $E(K_{\infty})[p^{\infty}]=\{0\}$, $\dim_{\Fp}(E(K_n)/p)=\rank(E(K_n))$. So we get $\rank(E(K_n)) \geq p^n$ as desired.
\end{proof}

We can now finally prove our theorem

\begin{theorem}
Assume that $(E,p)$ satisfies $(\star)$ and that $p$ does not divide $y_K$ in $E(K)$. Then we have
\begin{enumerate}[(i)]
\item $\Selinf(E/K_{\infty})^{\dual}$ is a free $\Lambda$-module of rank 1
\item $\rank(E(K_n))=p^n$ for all $n \geq 0$
\item $\Sha(E/K_n)[p^{\infty}]=\{0\}$ for all $n \geq 0$.
\end{enumerate}
\end{theorem}
\begin{proof}
First we prove (ii) and (iii). Let $n \geq 0$ be an integer. According to theorem \ref{Selmer_structure_theorem}, $\Sel(E/K_{\infty})^{\dual}$ is a free $\overbar{\Lambda}$-module of rank 1. Therefore $\dim_{\Fp}(\Sel(E/K_{\infty})^{\Gamma_n})=p^n$. Since by \cite{Matar1} corollary 2.4 $E(K_{\infty})[p^{\infty}]=\{0\}$, therefore it follows that $\Sel(E/K_n)$ injects via the restriction map into $\Sel(E/K_{\infty})^{\Gamma_n}$. So $\dim_{\Fp}(\Sel(E/K_n)) \le p^n$. Also, $E(K_{\infty})[p^{\infty}]=\{0\}$ implies that $\dim_{\Fp}(E(K_n))/p=\rank(E(K_n))$. Combining this with the result of the previous proposition, we get $\dim_{\Fp}(E(K_n)/p)=\rank(E(K_n)) \geq p^n$. But $E(K_n)/p$ injects via the Kummer map into $\Sel(E/K_n)$ and so by all the facts just mentioned, we must have
\begin{equation}\label{rank_equality}
\rank(E(K_n))=\dim_{\Fp}(E(K_n)/p)=\dim_{\Fp}(\Sel(E/K_n))=p^n
\end{equation}
Then (\ref{rank_equality}) proves (ii) and combining it with the following exact sequence gives (iii).

$$0 \longrightarrow E(K_n)/p \longrightarrow \Sel(E/K_n) \longrightarrow \Sha(E/K_n)[p] \longrightarrow 0$$\\
We now prove (i). Let $X=\Selinf(E/K_{\infty})^{\dual}$ and $D$ the torsion $\Lambda$-submodule of $X$. We have an exact sequence

$$0 \longrightarrow D \longrightarrow X \longrightarrow X/D \longrightarrow 0$$\\
The snake lemma with the multiplication by $p$ map gives the following exact sequence

$$(X/D)[p] \longrightarrow D/p \longrightarrow X/p$$\\
But $X/D$ is torsion-free and so $(X/D)[p]=\{0\}$. Therefore, $D/p$ injects into $X/p$. We claim that $D/p$ is finite.  To see this, note that from \cite{NSW} propositions 5.1.7, 5.1.8 and 5.1.9, we have that $D/p$ is infinite if and only if $X$ has positive $\mu$-invariant. According to \cite{Matar2} theorem 3.4 (and the remark after theorem \ref{direct_limit_theorem}), $X$ has $\mu$-invariant equal to zero and so it follows that $D/p$ is finite.

Now since $E(K_{\infty})[p^{\infty}]=\{0\}$ (\cite{Matar1} corollary 2.4), therefore we have an isomorphism $\Sel(E/K_{\infty}) \isomarrow \Selinf(E/K_{\infty})[p]$. So $X/p=\Sel(E/K_{\infty})^{\dual}$ and this is a free $\overbar{\Lambda}$-module of rank 1 by theorem \ref{Selmer_structure_theorem}. Since $X/p$ is a torsion-free $\overbar{\Lambda}$-module, it follows that $D/p=0$ which by Nakayama's lemma implies that $D=0$ i.e. $X$ is torsion-free. Moreover, as $X/p$ is a free $\overbar{\Lambda}$-module of rank 1, it follows that $\dim_{\Fp}(X/(\gamma-1, p)X)=1$ so by Nakayama's lemma $X$ is a cyclic $\Lambda$-module. Since it is also torsion-free, it must be free.
\end{proof}


\begin{thebibliography}{99}

    \bibitem{Bertolini} M.~Bertolini, \emph{Selmer groups and Heegner points in anticyclotomic $\Zp$-extensions}, Compositio Math. \textbf{99} (1995), 153-182

    \bibitem{BD} M.~Bertolini, H.~Darmon, \emph{Kolyvagin's descent and Mordell-Weil groups over ring class fields}, J. Reine Angew. Math.
        \textbf{412} (1990), 63-74

    \bibitem{BDCT} C.~Breuil, B.~Conrad, F.~Diamond, R.~Taylor, \emph{On the modularity of elliptic curves over $\Q$: Wild 3-adic exercises}, J. Amer. Math. Soc. \textbf{14} (2001), 843-939.

    \bibitem{Cornut} C.~Cornut, \emph{Mazur's conjecture on higher Heegner points}, Invent. Math. \textbf{148} (2002), 495-523.

    \bibitem{Gb_IIC} R.~Greenberg,  \emph{Introduction to Iwasawa theory for elliptic curves}, IAS/Park City Math. Ser. 9, Amer. Math Soc. Providence, 2001, pp. 407-464.

    \bibitem{Gross} B.~Gross, \emph{Kolyvagin's work on modular elliptic curves}, L-functions and arithmetic, 235-256, London Math. Soc. Lecture Series, \textbf{153}, 1989.

    \bibitem{Kolyvagin} V.~Kolyvagin, \emph{Euler Systems}, The Grothendieck Festschrift, Progr. in Math. 87, Birkh\"{a}user Boston, Boston, MA (1990).

    \bibitem{Manin} Y.I.~Manin \emph{Cyclotomic fields and modular curves}. Russian Math. Surveys \textbf{26}(6) 1971, 7-78.

    \bibitem{Matar1} A.~Matar, \emph{Selmer groups and anticyclotomic $\Zp$-extensions}, Math. Proc. Camb. Phil. Soc. \textbf{161}(3) (2016), 409-433

    \bibitem{Matar2} A.~Matar, \emph{Fine Selmer Groups, Heegner points and Anticyclotomic $\Zp$-extensions}, arxiv.org/abs/1503.06463 , to appear in the International Journal of Number Theory

    \bibitem{Mazur} B.~Mazur, \emph{Rational points of abelian varieties with values in towers of number fields}, Invent. Math \textbf{18} (1972), 183-266.

    \bibitem{NSW} J.~Neukirch, A.~Schmidt, K.~Wingberg, \emph{Cohomology of Number Fields}, second edition, Grundlehren der Mathematischen Wissenschaften \textbf{323}, Springer, 2008, xvi+825.

    \bibitem{PR} B.~Perrin-Riou, \emph{Fonctions L p-adiques, th\'{e}orie d'Iwasawa et points de Heegner}, Bull. Soc. Math. France \textbf{115} (1987), 399-456

    \bibitem{Wiles} A.~Wiles, \emph{Modular elliptic curves and Fermat's last theorem}, Ann. of Math. \textbf{141} (2) (1995), 443-551.

\end{thebibliography}
\end{document}